\documentclass[11pt,a4paper]{article}

\usepackage[T1]{fontenc} 
\usepackage[utf8]{inputenc} 
\usepackage[frenchb,english]{babel} 
\usepackage{fourier}
\usepackage[scaled=0.875]{helvet}
\usepackage{courier}

\usepackage{amsmath,amsthm,amssymb,amsfonts,epic,latexsym,hyperref,natbib
}
\usepackage[dvips]{graphicx}
\usepackage[usenames]{color}

\newtheorem{theo}{Theorem}[section]
\newtheorem{defi}[theo]{Definition}
\newtheorem{prop}[theo]{Proposition}

\newtheorem{assu}{Assumption}

\newtheorem{exam}{Example}

\newtheorem{rema}[theo]{Remark}

\def\R{\mathbb{R}}
\def \N{\mathbb{N}}
\def \Z{\mathbb{Z}}

\def\1{\mathbf{1}}
\def \PP{\mathbf{P}} 
\def\EE{\mathbf {E}} 
\def \E{\mathbb{E}} 
\def \Es{\mathbb{E}^\star}
\def \Et{\mathbb{E}^\theta}

\def \P{\mathbb{P}} 
\def \Ps{\mathbb{P}^\star}

\newcommand{\ts}{{\theta^{\star}}}
\newcommand{\Vs}{{\mathcal{V}^\star}}
\newcommand{\htet}{\widehat \theta_n}
\newcommand{\PPt}{\mathbf{P}^{\theta}}

\newcommand{\PPs}{\mathbf{P}^{\star}}

\newcommand{\EEs}{\mathbf{E}^{\star}}

\def \eps{\varepsilon}
\def\w{\omega}
\def\t{\theta}

\def\to{\rightarrow}

\newcommand{\argmax}{\mathop{\rm Argmax}}

\definecolor{pink}{rgb}{1,0.3,0.8}

\def\dd{\mathrm{d}}
\def\ee{\mathrm{e}}
\def\to{\rightarrow}

\def\Beta{\mathrm{B}}
\def\proj{\mathrm{proj}}

\newenvironment{merci}{\textbf{Acknowledgments.}}{ }

\begin{document}

\DeclareGraphicsExtensions{.pdf, .jpg, .jpeg, .png, .ps}

\title{
Asymptotic normality and efficiency of the maximum likelihood estimator for the
parameter of a ballistic random walk in a random environment}

\author{Mikael {\sc Falconnet}\footnote{Laboratoire Statistique et G\'enome, Universit\'e d'\'Evry Val
d'Essonne, UMR CNRS 8071, USC INRA,     E-mail:                          {\tt
  $\{$mikael.falconnet, catherine.matias$\}$@genopole.cnrs.fr}; 
$^\dag$  Laboratoire Analyse et
Probabilit\'es, Universit\'e d'\'Evry Val d'Essonne, 
E-mail: {\tt dasha.loukianova@univ-evry.fr}
}
\and 
Dasha {\sc Loukianova}$^{\dag}$ 
\and 
Catherine {\sc Matias}$^{*}$
}

\maketitle

\begin {abstract}
We  consider a  one dimensional  ballistic random  walk evolving  in a
parametric    independent   and    identically    distributed   random
environment. We study the asymptotic properties of 
the maximum likelihood estimator of the parameter based on a single observation
of the  path till  the time it  reaches a  distant site.  We  prove an
asymptotic  normality  result for  this  consistent  estimator as  the
distant  site tends  to infinity  and establish  that it  achieves the
Cramér-Rao bound.  We also explore in a simulation setting the numerical behaviour of
asymptotic confidence regions for the parameter value.
\end{abstract}


\section{Introduction}
Random walks in random  environments (RWRE) are stochastic models that
allow two kinds of uncertainty  in physical systems: the first one is
due to the heterogeneity of the environment, and the second one to the
evolution of a particle in  a given environment.  The first studies of
one-dimensional RWRE were done by~\cite{Chernov} with a model of DNA 
replication, and   by~\cite{Temkin}   in   the  field   of
metallurgy.  From the latter work, the random media literature inherited some famous
terminology such as \textit{annealed} or \textit{quenched}
law. The  limiting behaviour of the  particle in \citeauthor{Temkin}'s
model  was successively  investigated by  \cite{Kozlov,Sol}
and~\cite{KKS}.  Since these  pioneer works on one-dimensional RWRE,
the related  literature in physics  and probability theory  has become
richer and  source of fine  probabilistic results that the  reader may
find in recent surveys including~\cite{Hughes} and~\cite{ZeitouniSF}.

 The  present   work  deals  with  the  one-dimensional   RWRE  where  we
 investigate a  different kind of question than  the limiting behaviour
 of the  walk. We adopt a statistical  point of view
 and are interested in inferring the distribution of the environment
 given the observation  of a long trajectory of  the random walk.  This
 kind of questions has already been studied in the context of random 
 walks in random colorings of $\Z$~\citep{BK,Matzinger,LM} 
as well as in the  context of RWRE for a characterization of the environment
  distribution~\citep{AdEn,Comets_etal}.  Whereas~\citeauthor{AdEn} deal
  with  very  general  RWRE  and   present  a  procedure  to  infer  the
  environment distribution through a system of moment equations,
  \citeauthor{Comets_etal} provide a maximum likelihood estimator (MLE)
  of the parameter of the environment distribution in the 
  specific  case   of  a   transient   ballistic  one-dimensional
  nearest neighbour path. In the latter work, the authors establish the consistency
  of their estimator and provide synthetic experiments to assess its 
  effective performance. It turns out that this 
  estimator   exhibits   a   much   smaller  variance   than   the   one
 of~\citeauthor{AdEn}.   We  propose to  establish  what the  numerical
 investigations  of~\citeauthor{Comets_etal}  suggested,  that is,  the
 asymptotic normality of the MLE  as  well as
 its asymptotic efficiency (namely, that it asymptotically achieves the Cramér-Rao bound).

This article is organised as follows.  In Section~\ref{sect:RWRE}, we
introduce the framework of the one dimensional  ballistic random walk
in       an      independent   and    identically    distributed  (i.i.d.)      parametric      environment.        In
Section~\ref{sect:M_Estimator}, we present the MLE procedure
developed  by~\citeauthor{Comets_etal} to infer  the parameter  of the
environment distribution. 
Section~\ref{sect:BP}  recalls  some   already  known  results  on  an
underlying branching  process in a  random environment related  to the
RWRE.  
Then, we state in Section~\ref{sect:res} our asymptotic normality result in the wake of
additional hypotheses required to prove it and listed in 
Section~\ref{sec:more_hyps}.   In Section~\ref{sect:ex},  we present
 three examples  of  environment distributions  which are  already
introduced  in~\cite{Comets_etal}, and  we check  that  the additional
required assumptions  of Section~\ref{sec:more_hyps} are  fulfilled, so
that the MLE is asymptotically normal and
efficient in these cases.  The proof of the asymptotic normality result is presented
in Section~\ref{sect:CLT}.  We apply to the score vector sequence a central
limit    theorem    for    centered   square-integrable    martingales
(Section~\ref{sec:CLT_dotphi})  and   we  adapt  to   our  context  an
asymptotic  normality result for  M-estimators (Section~\ref{sec:AN}).
To  conclude this  part, we  provide in  Section~\ref{sec:nonzero} the
proof of a sufficient condition for the non-degeneracy of the Fisher information.
Finally, Section~\ref{sect:simus} illustrates our results on synthetic
data by exploring empirical coverages of asymptotic confidence regions.

\section{Material and results} 


\subsection{Properties of a transient random walk in a random
  environment} \label{sect:RWRE} 
Let us introduce a one-dimensional random walk (more precisely a
nearest  neighbour path) evolving  in a  random environment  (RWRE for
short) and recall its elementary properties. 
We start by considering the environment defined through the collection
$\w=(\w_x)_{x\in\Z}\in   (0,1)^{\mathbb{Z}}$   of i.i.d.  random variables, 
with parametric  distribution $\nu=\nu_{\theta}$ that  depends on some
unknown parameter $\theta\in\Theta$. 
We further assume that $\Theta \subset \R^d$ is a compact set.  We let
$\P^{\theta}=\nu_{\theta}^{\otimes \Z}$ be the law on 
$(0,1)^{  \Z}$  of  the  environment  $\w$ and  $\E^{\theta}$  be  the
corresponding expectation. 

Now,  for  fixed environment  $\w$,  let  $X=(X_t)_{t\in\Z_+}$ be  the
Markov 
chain on  $\Z$ starting at  $X_0=0$ and with  (conditional) transition
probabilities 
\[
  P_{\w}(X_{t+1}=y|X_t=x)=\left \{\begin{array}{lr}
\w_x&\mbox{if}\ y=x+1,\\
1-\w_x&\mbox{if}\ y=x-1,\\
0&\mbox{otherwise}.
\end{array}\right.
\]
The \emph{quenched}  distribution $P_{\w}$ is  the conditional measure
on the  path space of  $X$ given $\w$.  Moreover,  the \emph{annealed}
distribution of $X$ is given by 
\[
  \PP^{\theta}(\cdot)=\int P_{\omega}(\cdot)\dd\P^{\theta}(\omega).
\]
 We  write  $E_{\w}$  and   $\EE^{\theta}$  for  the
corresponding quenched and annealed expectations, respectively.
In the following, we assume that the process $X$ is generated under the
true parameter value  $\ts$, an interior point of  the parameter space
$\Theta$, that we aim at  estimating. We shorten to 
$\PPs$   and   $\EEs$   (resp.   $\Ps$   and   $\Es$)  the   annealed
probability $\PP^{\ts}$ and its corresponding expectation~$\EE^{\ts}$ (resp. the law of the environment $\P^{\ts}$ and its corresponding expectation~$\E^\ts$) under
parameter value~$\ts$.

The behaviour of the process $X$ is related to the ratio sequence 
\begin{equation}
  \label{eq:rho}
  \rho_x=\frac{1-\w_x}{\w_x}, \qquad x \in {\mathbb Z}.
\end{equation}
We refer to~\cite{Sol} for the classification of $X$ between transient
or  recurrent cases  according  to whether  $\E^{\t}(\log \rho_0)$  is
different or  not from zero~\cite[the classification  is also recalled
in][]{Comets_etal}. In our setup,  we consider a transient process and
without loss of generality, assume that it is transient to the right,
thus corresponding to $\E^{\t}(\log \rho_0)<0$.  
The transient  case may  be further split  into two  sub-cases, called
\emph{ballistic} and \emph{sub-ballistic}  that correspond to a linear
and  a sub-linear speed  for the  walk, respectively.  More precisely,
letting $T_n$ 
be the first hitting time of the positive integer $n$, 
\begin{equation} \label{equa:HittingTime}
T_n = \inf \{ t \in \Z_+ \, : \, X_t = n \},
\end{equation}
and   assuming  $\E^{\theta}(\log\rho_0)<0$   all   through,  we   can
distinguish the following cases 
\begin{itemize}
\item[(a1)]  (Ballistic).  If   $\E^{\theta}(\rho_0)  <  1$,  then,  $
  \PP^{\theta} \mbox{-almost surely},$ 
  \begin{equation}
\label{eq:lln1}
\frac{T_n}{n}           \xrightarrow[n          \to          \infty]{}
\frac{1+\E^{\theta}(\rho_0)}{1-\E^{\theta}(\rho_0)}. 
\end{equation}
\item[(a2)]  (Sub-ballistic). If  $\E^{\theta}(\rho_0)  \geq 1$,  then
  $T_n / n \to + \infty$, 
$ \PP^{\theta} \mbox{-almost surely}$ when $n$ tends to infinity.
  \end{itemize}

Moreover, the  fluctuations of $T_n$  depend in nature on  a parameter
$\kappa \in (0, \infty]$, 
which is defined as the  unique positive solution of 
\[
\E^{\theta}(\rho_0^\kappa)=1
\]
when such a number exists, and $\kappa=+\infty$ otherwise. The ballistic case 
corresponds to $\kappa >1$.
Under mild additional assumptions,   \cite{KKS}  proved that 

 \begin{itemize}
\item[(aI)]  \label{item:aI}  if  $\kappa  \geq  2$,  then  $T_n$  has
  Gaussian fluctuations. Precisely, if $c$ denotes the limit in 
(\ref{eq:lln1}), then
$n^{-1/2}(T_n-nc)$ when $\kappa>2$, and $(n \log n)^{-1/2}(T_n-nc)$ when $\kappa=2$
have a non-degenerate Gaussian limit.
\item[(aII)] \label{item:aII} if $\kappa < 2$, then $n^{-1/\kappa}(T_n-d_n)$
has a  non-degenerate limit distribution,  which is a stable  law with
index $\kappa$. \\ 
The centering is 
$d_n=0$ for $\kappa <1$, $d_n=a n  \log n$ for $\kappa =1$, and $d_n=a
n$ for $\kappa \in (1,2)$, 
for some positive constant $a$.
  \end{itemize}


\subsection{A consistent estimator} \label{sect:M_Estimator}
We briefly recall the definition of the estimator proposed in~\cite{Comets_etal} to
infer the parameter $\theta$,  when we observe $X_{[0,T_n]}=(X_t \, :
\, t=0,1,\ldots, T_n)$, for some value $n \geq 1$.  It is defined as the
maximizer  of  some  well-chosen  criterion  function,  which  roughly
corresponds to the log-likelihood of the observations.  

We start by introducing the statistics $(L_x^n)_{x\in\Z}$, defined as 
\[
L_x^n := \sum_{s=0}^{T_n-1}\1_{\{X_s=x;\ X_{s+1}=x-1\}}, 
\]
namely  $L_x^n$   is  the  number   of  left  steps  of   the  process
$X_{[0,T_n]}$  from  site~$x$.   Here,  $\1_{\{\cdot\}}$  denotes  the
indicator function.

\begin{defi}
Let   $\phi_{\theta}$  be the  function  from $\Z_+^2$ to $\R$ given by 
\begin{equation} \label{eq:PhiTheta}
  \phi_{\theta}(x,y)=\log\int_0^1a^{x+1}(1-a)^{y}\dd\nu_{\theta}(a).
\end{equation}
The criterion function $\t \mapsto\ell_n(\t)$ is defined as 
\begin{equation}\label{eq:l}
\ell_n(\theta)=\sum_{x=0}^{n-1}  \phi_{\theta}(L_{x+1}^n,L_{x}^n). 
\end{equation}
\end{defi}
We  now recall the  assumptions stated  in~\cite{Comets_etal} ensuring
that the maximizer of criterion  $\ell_n$ is a consistent estimator of
the unknown parameter.  
\begin{assu}(Consistency conditions). \label{assu:consistance}
  \begin{itemize}
  \item [i)] \label{assu:transience} (Transience to the right).  
For any $ \theta\in\Theta, \E^{\theta}|\log\rho_0|<\infty$
and~$\E^{\theta}(\log\rho_0)<0$.  
\item[ii)] \label{assu:ballistic} (Ballistic case). 
For any $ \theta\in\Theta, \   \E^{\t}(\rho_0)<1$.
\item[iii)]\label{assu:cont} (Continuity). 
For  any $(x,y)\in  \Z_+^2$, the  map $\t  \mapsto
\phi_{\theta}(x,y)$ is continuous on the parameter set $\Theta$.
 \item[iv)] \label{assu:ident} (Identifiability). For any 
 $ (\theta,\theta')\in\Theta^2,$ 
 $\nu_{\theta}\neq\nu_{\theta'}\iff \theta\neq\theta'.$
\item [v)] \label{assu:BoundedBelow}
The  collection of  probability measures  $\{ \nu_\t  \, :  \,  \t \in
\Theta \}$ is such that 
\[
   \inf_{\theta \in \Theta} \E^\t [\log (1- \w_0)] > -\infty.
\]
  \end{itemize}
\end{assu}
According to Assumption~\ref{assu:consistance} 
point~\hyperref[assu:cont]{$iii)$},    the   function   $\theta\mapsto
\ell_n(\theta)$ is  continuous on the compact  parameter set $\Theta$.
Thus, it achieves its maximum, and the estimator $\htet$ is defined as
one maximizer of this criterion.  
\begin{defi}
An estimator $ \htet$ of $\theta$ is defined as a measurable choice
\begin{equation}\label{eq:estimator}
\htet \in \argmax_{\theta\in\Theta}\ell_n(\theta).
\end{equation}
\end{defi}
Note   that  $\htet$   is   not  necessarily   unique.  As   explained
in~\cite{Comets_etal}, with a slight abuse of notation, $\htet$ may be
considered as MLE.  Moreover, under
Assumption~\ref{assu:consistance},  \cite{Comets_etal}  establish  its
consistency, namely its convergence  in $\PPs$-probability to the true
parameter value $\ts$.

\subsection{The role of an underlying branching process}\label{sect:BP}
We  introduce in  this section  an underlying  branching  process with
immigration in random environment (BPIRE) that is naturally related to the RWRE.
Indeed, it is well-known that for an i.i.d. environment, 
under the annealed law $\PPs$, the sequence 
$L_n^n,L_{n-1}^n,\dots,  L^n_0$ has  the same  distribution as  a BPIRE
denoted $Z_0,\dots, Z_n,$ and defined by 
\begin{equation} \label{eq:BPRE}
Z_0=0, \quad \mbox{and for } k=0,\dots,n-1,\quad
Z_{k+1}=\sum_{i=0}^{Z_k}\xi'_{k+1,i},
\end{equation}
with $(\xi'_{k,i})_{k\in\N;i\in\Z_+}$ independent and
\[
  \forall m\in\Z_+, \quad P_{\w}(\xi'_{k,i}=m)=(1-\w_{k})^m\w_{k} ,
\]
\cite[see for instance][]{KKS,Comets_etal}.  
Let   us   introduce   through   the  function   $\phi_{\t}$   defined
by~\eqref{eq:PhiTheta} the transition kernel~$Q_\t$ on $\Z_+^2$ defined as 
\begin{equation}
  \label{eq:Qtransition}
Q_{\t}(x,y)=\binom{x+y}{x}\ee^{\phi_\t(x,y)}=  \binom{x+y}{x} \int_0^1
a^{x+1}(1-a)^y \dd\nu_{\t}(a). 
\end{equation}
Then for  each value  $\t \in \Theta$,  under annealed law  $\PPt$ the
BPIRE  $(Z_n)_{n\in\Z_+}$  is   an  irreducible  positive  recurrent
homogeneous Markov chain with transition
kernel~$Q_{\t}$ and a unique stationary probability distribution denoted
by $\pi_\t$.  Moreover, the moments of $\pi_\t$ may be characterised
through the distribution of the ratios $(\rho_x)_{x\in \Z}$. The following statement is a direct    consequence    from     the    proof    of    Theorem    4.5
in~\cite{Comets_etal} (see Equation~(16) in this proof).
\begin{prop}\cite[Theorem 4.5 in][]{Comets_etal}.
\label{prop:moment_pi}
  The invariant probability measure $\pi_\t$ is positive on $\Z_+$ and satisfies 
\[
\forall j\ge 0, \quad \sum_{k\ge j+1} k(k-1)\dots(k-j) \pi_\t(k) = (j+1)! \ \Et
\Big[\Big(\sum_{n\ge 1} \prod_{k=1}^n \rho_k\Big) ^{j+1}\Big].
\]
In particular, $\pi_\t$ has a finite first moment in the ballistic case.
\end{prop}
Note that the criterion $\ell_n$ satisfies the following property 
\begin{equation}\label{eq:logvraisannealed}
\ell_n(\theta)
\sim \sum_{k=0}^{n-1}\phi_{\theta}(Z_k,Z_{k+1}) \mbox{ under } \PPs,
\end{equation}
where $\sim $  means equality in distribution. For  each value $\t \in
\Theta$,     under     annealed     law     $\PPt$     the     process
$((Z_n,Z_{n+1}))_{n\in\Z_+}$ is also an irreducible positive 
recurrent homogeneous Markov  chain with unique stationary probability
distribution denoted by $\tilde \pi_\t$ and defined as
\begin{equation}\label{eq:pi_tilde}
\tilde  \pi_\t (x,y) =  \pi_\t(x) Q_\t(x,y),  \quad \forall  (x,y) \in
\Z_+^2. 
\end{equation}
For any function  $g:\Z_+^2 \to \R$ such that  $ \sum_{x,y}\tilde \pi_\t
(x,y) |g(x,y)| < \infty$, we  denote by $\tilde \pi_{\t}(g)$ the quantity
\begin{equation}\label{eq:pi_tilde_g}
\tilde \pi_\t(g) = \sum_{(x,y)\in \Z_+^2}\tilde \pi_\t (x,y) g(x,y).
\end{equation}
We extend the notation  above for any function $g=(g_1,\dots,g_d):\Z_+^2
\to \R^d$ such that $ \tilde \pi_\t( \|g\|) <\infty$, where $\| \cdot
\|$ is the uniform norm, and denote by $\tilde \pi_\t(g)$ the
vector  $(\tilde  \pi_\t(g_1),\dots,\tilde  \pi_\t(g_d))$.  The  following
ergodic theorem is valid.

\begin{prop}\cite[Theorem 4.2 in Chapter 4 from][]{Revuz}. 
\label{prop:lln}
Under               point~\hyperref[assu:transience]{$i)$}        in
Assumption~\ref{assu:consistance}, for any function $g : \Z_+^2\to 
\R^d$, such that $ \tilde \pi_\t( \|g\|) <\infty$ the following ergodic
theorem holds 
\[
  \lim_{n\to\infty}\frac 1 n \sum_{k=0}^{n-1} g(Z_k,Z_{k+1})
= \tilde \pi_\t( g) ,
\]
$\PPt$-almost surely and in $ \mathbb{L}^1(\PPt)$. 
\end{prop}


\subsection{Assumptions for asymptotic normality}\label{sec:more_hyps}
Assumption~\ref{assu:consistance} is  required  for  the  construction of  a
consistent estimator of the parameter $\theta$. It mainly consists in
a transient random walk with linear speed (ballistic regime) plus some
regularity    assumptions    on   the    model    with   respect    to
$\theta\in\Theta$. Now, asymptotic  normality result for this estimator
requires additional hypotheses.  

In the following,  for any function $g_\t$ depending  on the parameter
$\t$, the symbols $\dot  g_\t$ or $\partial_\t g_\t$ and $\ddot{g}_\t$
or  $\partial_\t^2  g_\t$  denote  the (column)  gradient  vector  and
Hessian   matrix  with  respect   to  $\t$,   respectively.  Moreover,
$Y^\intercal$  is the row  vector obtained  by transposing  the column
vector $Y$.

\begin{assu} (Differentiability). \label{assu:diff}
The collection of probability measures $\{ \nu_\t \, : \, \t \in\Theta
\}$ is such that for any $(x,y)\in \Z_+^2$, the map $\t \mapsto 
\phi_{\theta}(x,y)$ is twice continuously differentiable on $\Theta$.
\end{assu}

\begin{assu} (Regularity conditions).\label{assu:smoothness}
For any $\t \in \Theta$, there exists some $q>1$ such that
\begin{equation}\label{eq:dotphi_stat_integr}
\tilde \pi_\t\Big(\|\dot \phi_{\t}\|^{2q}\Big)<+\infty.
\end{equation}
For any $x\in\Z_+,$
\begin{equation}\label{eq:chaing_dot}
\quad \sum_{y\in \Z_+} \dot Q_{\t}(x,y)=\partial_\t \sum_{y\in \Z_+}  Q_{\t}(x,y)=0.
\end{equation}
\end{assu}

\begin{assu} (Uniform conditions).\label{assu:uniform}
For   any   $\t   \in   \Theta$,  there   exists   some   neighborhood
$\mathcal{V}(\t)$ of $\t$ such that 
\begin{equation} \label{eq:phi_dot_square_control} 
 \tilde   \pi_\t  \Big(\sup_{\t'   \in   \mathcal{V}(\t)}  \|\dot
 \phi_{\t'} \|^2 \Big) < +\infty
\quad \mbox{and} \quad 
  \tilde  \pi_\t \Big( \sup_{\t' \in 
   \mathcal{V}(\t)} \|\ddot \phi_{\t'} \|\Big) < +\infty.
\end{equation}
\end{assu}

Assumptions~\ref{assu:diff} and~\ref{assu:smoothness} are technical and
involved  in  the proof  of  a central  limit  theorem  (CLT) for  the
gradient vector of the criterion~$\ell_n$, also called score
vector sequence. Assumption~\ref{assu:uniform} is also technical and involved
in  the proof  of asymptotic  normality  of $\widehat  \t_n$ from  the
latter CLT. Note that Assumption~\ref{assu:smoothness} also allows us to define the
matrix 
\begin{equation}
  \label{eq:Sigma}
   \Sigma_{\t} 
=   \tilde   \pi_{\t}  \Big(\dot   \phi_{\t}^{\phantom{\intercal}}\dot
\phi_{\t}^\intercal \Big). 
\end{equation}
Combining                                                           the
definitions~\eqref{eq:Qtransition},\eqref{eq:pi_tilde},\eqref{eq:pi_tilde_g}
and~\eqref{eq:Sigma} with Assumption~\ref{assu:smoothness}, we obtain the equivalent expression for $\Sigma_{\t}$
\begin{align}
\Sigma_{\t} = & \sum_{x \in\Z_+}\sum_{y\in \Z_+} \pi_{\t}(x) \frac 1 { Q_\t(x,y)} \dot{Q}_\t
(x,y) \dot Q_\t (x,y)^\intercal \nonumber \\
= &-\sum_{x \in\Z_+}\sum_{y\in \Z_+} \pi_{\t}(x)\left ( \ddot{Q}_\t(x,y) - \frac 1 { Q_\t(x,y)} \dot{Q}_\t
(x,y) \dot Q_\t (x,y)^\intercal\right ) \nonumber \\
= &   - \tilde \pi_\t (\ddot \phi_{\t}) .\label{eq:Sigma_v2}
 \end{align}
\begin{assu}(Fisher information matrix).\label{assu:nonsingular}
  For  any value  $\t \in  \Theta$,  the matrix  $\Sigma_{\t}$ is  non
  singular.  
\end{assu}
Assumption~\ref{assu:nonsingular} 
states invertibility of  the Fisher information matrix $\Sigma_{\ts}$.
This assumption is necessary to prove asymptotic 
normality of $\widehat \t_n$ from  the previously mentioned CLT on the
score vector sequence.


\subsection{Results} \label{sect:res}

\begin{theo}\label{thm:CLT_dotphi}
Under Assumptions~\ref{assu:consistance} to~\ref{assu:smoothness}, the 
score vector sequence $ \dot \ell_n(\ts)/\sqrt{n}$ is asymptotically
normal with mean zero and finite covariance matrix $ \Sigma_{\ts}$.
\end{theo}

\begin{theo} (Asymptotic normality). \label{theo:CLT}
Under   Assumptions~\ref{assu:consistance}  to~\ref{assu:nonsingular},
for any  choice of $\widehat\theta_n$ satisfying~\eqref{eq:estimator},
the  sequence  $\{\sqrt{n}  (\htet-\ts)\}_{n  \in  \Z_+}$  converges  in
$\PPs$-distribution  to   a  centered  Gaussian   random  vector  with
covariance  matrix  $   \Sigma_{\ts}  ^{-1}$.   
\end{theo}

Note  that  the limiting  covariance  matrix  of  $\sqrt{n} \htet$  is
exactly the inverse Fisher information matrix of the model.  As such, our
estimator is  efficient. Moreover, the previous theorem may
be used to build asymptotic confidence regions for~$\theta$, as
illustrated in Section~\ref{sect:simus}. Proposition~\ref{prop:sigma_approx} below explains how to estimate the Fisher information matrix $\Sigma_\ts$. Indeed, $\Sigma_\ts$ is defined via the invariant distribution $\tilde \pi_\ts$ which possesses no analytical expression. To bypass the problem, we rely on the \emph{observed Fisher information matrix} as an estimator of  $\Sigma_\ts$.

\begin{prop} \label{prop:sigma_approx}
Under Assumptions~\ref{assu:consistance} to~\ref{assu:nonsingular}, the observed information matrix 
\begin{equation} \label{equa:Sigma_hat}
  \hat    \Sigma_{n}    = -   \frac    1    n    \sum_{x=0}^{n-1}    \ddot
  \phi_{\htet}^{\phantom{\intercal}} (L_{x+1}^n,L_x^n)
\end{equation}
converges in $\PPs$-probability to $\Sigma_\ts$.
\end{prop}

\begin{rema}
We observe  that the fluctuations of  the estimator $\widehat\theta_n$
are unrelated to those of $T_n$ or those of $X_t$, see
\hyperref[item:aI]{(aI)}-\hyperref[item:aII]{(aII)}. Though there is a
change of limit law from Gaussian to stable as $\E^{\theta}(\rho_0^2)$
decreases from larger to smaller than~$1$, the MLE remains asymptotically normal in the full ballistic region
(no  extra  assumption   is  required  in  Example~\ref{ex:deuxpoints}
introduced  in  Section~\ref{sect:ex}). We  illustrate  this point  by
considering a naive estimator at the end of Subsection~\ref{sect:ex1}.
\end{rema}

We conclude this section by providing a sufficient condition for
Assumption~\ref{assu:nonsingular}  to be  valid, namely  ensuring that
$\Sigma_\t$ is positive definite.

\begin{prop}\label{prop:Sigma_non_zero}
For the covariance matrix $ \Sigma_\t$ to be positive definite, it is
sufficient that the linear span in ${\mathbb R}^d$ of the gradient vectors 
$\dot \phi_{\t}(x,y)$, with $(x,y) \in\Z_+^2$ is equal to the full
space, or equivalently,  that
\[
{\rm Vect}\Big\{ \partial_\t
\Et(\omega_0^{x+1}(1-\omega_0)^y) \, : \, (x,y) \in \Z_+^2
\Big\} ={\mathbb R}^d. 
\]
\end{prop}

Section~\ref{sect:CLT}   is
devoted    to    the    proof    of    Theorem~\ref{theo:CLT}    where
Subsections~\ref{sec:CLT_dotphi},~\ref{sec:approx} and~\ref{sec:nonzero}   are concerned with the       proofs        of       Theorem~\ref{thm:CLT_dotphi}, Proposition~\ref{prop:sigma_approx} and 
Proposition~\ref{prop:Sigma_non_zero},  respectively.


\section{Examples} \label{sect:ex}

\subsection{Environment with finite and known support} \label{sect:ex1}

\begin{exam} \label{ex:deuxpoints}
Fix $a_1 < a_2 \in (0,1)$ and let
$\nu_p=p\delta_{a_1}+(1-p)\delta_{a_2}$, where  $\delta_a$ is the Dirac
mass located at value $a$. Here, the unknown parameter is the proportion
$p\in  \Theta  \subset  [0,1]$  (namely~$\theta=p$). We  suppose  that
$a_1$,      $a_2$      and       $\Theta$      are      such      that
points~\hyperref[assu:transience]{$i)$} 
and~\hyperref[assu:ballistic]{$ii)$} in
Assumption~\ref{assu:consistance} are satisfied.  
\end{exam}

This example  is easily generalized  to $\nu$ having $m\ge  2$ support
points namely $ \nu_\t=\sum_{i=1}^m p_ia_i$, where $a_1,\dots,a_m$ are
distinct,     fixed     and    known     in     $(0,1)$,    we     let
$p_m=1-\sum_{i=1}^{m-1}p_i$     and    the     parameter     is    now
$\theta=(p_1,\dots,p_{m-1})$.

In the framework of Example~\ref{ex:deuxpoints}, we have
\begin{equation} \label{equa:Phi2pts}
 \phi_{p}(x,y)      =\log    [    p    a_1^{x+1}(1-a_1)^y    +(1-p)
 a_2^{x+1}(1-a_2)^y ] ,
\end{equation}
and
\begin{equation}\label{eq:ell_2points}
  \ell_n(p):=\ell_n(\theta) = \sum_{x=0}^{n-1}\log \Big[ p
  a_1^{L_{x+1}^n+1}(1 \!  -\! a_1)_{\phantom{1}}^{L_x^n} + (1\! -\! p) 
  a_2^{L_{x+1}^n+1}(1\! -\! a_2)_{\phantom{2}}^{L_x^n} \Big].
\end{equation}

\cite{Comets_etal} proved that 
$
 \widehat p_n = \argmax_{p\in\Theta}\ell_n(p)
$ 
converges in $\PPs$-probability to $p^\star$. 
There is no  analytical expression for the
value  of  $\widehat  p_n$.  Nonetheless,  this estimator  may  be  easily
computed by  numerical methods. We now establish  that the assumptions
needed for asymptotic normality are also satisfied in this case, under
the only additional assumption that $\Theta\subset (0,1)$.  

 \begin{prop} \label{prop:2ptsNormal}
In  the framework  of  Example~\ref{ex:deuxpoints}, assuming  moreover
that     $\Theta    \subset     (0,1)$,    Assumptions~\ref{assu:diff}
to~\ref{assu:uniform} are satisfied. 
\end{prop}

\begin{proof}
The function $p \mapsto \phi_{p}(x,y)$ given by \eqref{equa:Phi2pts}
is twice continuously differentiable  for any $(x,y)$. The derivatives
are given by 
 \begin{align*}
  \dot \phi_{p}(x,y) &= \ee^{-\phi_{p}(x,y)} [a_1^{x+1}(1-a_1)^y
    -a_2^{x+1}(1-a_2)^y ], \\
    \ddot \phi_{p}(x,y)       &= - \dot \phi_{p}(x,y)^2. 
\end{align*}
Since $\exp  [  \phi_{p}(x,y)]\ge   p
a_1^{x+1}(1-a_1)^y $ and  $\exp  [  \phi_{p}(x,y)]\ge  (1- p)
a_2^{x+1}(1-a_2)^y $, we obtain the bounds
\begin{equation*}
|  \dot  \phi_{p}(x,y)  |  \le  \frac  1  {p}  +  \frac  1
{1-p} .
\end{equation*}
Now, under the additional assumption that $\Theta\subset (0,1)$, there
exists some $A\in (0,1)$ such that $\Theta\subset [A,1-A]$ and then 
\begin{equation}
  \label{eq:bound_dot_phi_p}
\sup_{(x,y)\in \Z_+^2} | \dot \phi_{p}(x,y) |\le \frac 2 A 
\quad \mbox{and} \quad
\sup_{(x,y)\in \Z_+^2} | \ddot \phi_{p}(x,y) |\le \frac{4}{A^2}, 
\end{equation}
which yields that \eqref{eq:dotphi_stat_integr} and 
\eqref{eq:phi_dot_square_control} are satisfied. 

Now, noting that 
\[
  \dot Q_{\t}(x,y)=\binom{x+y}x [a_1^{x+1}(1-a_1)^y-a_2^{x+1}(1-a_2)^y],
\]
and that 
\begin{equation} \label{eq:negbinom}
\sum_{y=0}^{\infty}\binom{x+y}xa^{x+1}(1-a)^y=1,  \quad \forall  x \in
\Z_+, \ \forall a \in (0,1), 
\end{equation}
yields \eqref{eq:chaing_dot}.
\end{proof}

\begin{prop} \label{prop:NonDegeEx1}
In the framework of Example~\ref{ex:deuxpoints}, the covariance matrix
$\Sigma_\t$        is         positive        definite,        namely
Assumption~\ref{assu:nonsingular} is satisfied.
\end{prop}

\begin{proof}[Proof of Proposition~\ref{prop:NonDegeEx1}]
We have 
\[
\E^p(\omega_0) = p(a_1-a_2) + a_2,
\]
with derivative $a_1-a_2  \neq 0$, which achieves the  proof thanks to
Proposition~\ref{prop:Sigma_non_zero}. 
\end{proof}

Thanks             to            Theorem~\ref{theo:CLT}            and
Propositions~\ref{prop:2ptsNormal}    and~\ref{prop:NonDegeEx1},   the
sequence  $\{\sqrt{n}   (\widehat  p_n  -   p^\star)\}$  converges  in
$\PPs$-distribution  to  a  non  degenerate centered  Gaussian  random
variable, with variance 
\[
\Sigma_{p^\star}^{-1}= \Big \{ \sum_{(x,y)\in \Z_+^2} \pi_{p^\star}(x)
\binom{x+y}{x} \frac
{[a_1^{x+1}(1-a_1)^y-a_2^{x+1}(1-a_2)^y          ]^2}         {p^\star
  a_1^{x+1}(1-a_1)^y+(1-p^\star )a_2^{x+1}(1-a_2)^y}
\Big\}^{-1}.
\]

\begin{rema}\label{ex:temkin}  \cite[Temkin model, cf.][]{Hughes}. With
  $a \in (1/2,1)$ known and $\theta=p \in (0,1)$ unknown, we consider 
  $\nu_\t=p \delta_a + (1-p)  \delta_{1-a}$. This is a particular case
  of 
Example~\ref{ex:deuxpoints}. It is easy  to see that transience to the
right and ballistic regime, respectively, are equivalent to 
\[
p >1/2, \qquad p >a,
\]
and that in the ballistic  case, the limit $c=c(p)$ in (\ref{eq:lln1})
is given by 
\[
c(p)= \frac{a+p-2ap}{(2a-1)(p-a)}.
\]
We construct  a new  estimator $\tilde  p_n$ of
$p$ solving the relation 
$c(\tilde p_n)=T_n/n$, namely
\[
\tilde p_n =  \frac{a}{2a-1} \times \frac{(2a-1)T_n+ {n}}{T_n+n}.
\]
This   new   estimator   is   consistent   in   the   full   ballistic
region. However, for all $a>1/2$ and $p>a$ 
but close to it, we have $\kappa \in (1,2)$, the fluctuations of $T_n$
are of order $n^{1/\kappa}$, 
and those of $\tilde p_n$ are of the same order. This new estimator is
much more spread out than the MLE $\widehat p_n$. 
\end{rema}

\subsection{Environment with two unknown support points} \label{sect:ex2}

\begin{exam} \label{ex:deuxpts_3param}
We let $\nu_\t=p\delta_{a_1}+(1-p)\delta_{a_2}$  and now the unknown
  parameter is  $\theta=(p,a_1,a_2) \in  \Theta$, where $\Theta$  is a
  compact subset of
\[
  (0,1)\times \{(a_1,a_2)\in (0,1)^2 \, : \,  a_1 < a_2\}.
\]
We       suppose      that       $\Theta$      is       such      that
points~\hyperref[assu:transience]{$i)$} and~\hyperref[assu:ballistic]{$ii)$}                                 in
Assumption~\ref{assu:consistance} are satisfied.  
\end{exam}

The  function $\phi_\t$  and the  criterion  $\ell_n(\cdot)$ are
given        by~\eqref{equa:Phi2pts}        and~\eqref{eq:ell_2points},
respectively.  \cite{Comets_etal} have established that  the
estimator $ \widehat \t_n $ is well-defined and consistent in probability. 
Once again, there is no  analytical expression for the
value  of  $\widehat  \t_n$.  Nonetheless,  this estimator  may  also be  easily
computed by numerical methods.   We now establish that the assumptions
needed for asymptotic normality are also satisfied in this case, under
a mild additional moment assumption.

\begin{prop} \label{prop:2pts3paraNormal}
In the framework of Example~\ref{ex:deuxpts_3param}, assuming moreover
that     $\Et(\rho_0^3)      <     1$,     Assumptions~\ref{assu:diff}
to~\ref{assu:uniform} are satisfied.  
\end{prop}

\begin{proof}
 In  the proof of  Proposition~\ref{prop:2ptsNormal}, we  have already
controled the derivative of  $\t \mapsto \phi_\t(x,y)$ with respect to
$p$.  Hence, it  is now  sufficient  to control  its derivatives  with
respect    to    $a_1$    and    $a_2$   to    achieve    the    proof
of~\eqref{eq:dotphi_stat_integr}
and~\eqref{eq:phi_dot_square_control}. We have
 \begin{align}\label{equa:Partial3param}    
\nonumber  \partial_{a_1} \phi_{\t}(x,y)     &    =
\ee^{-\phi_{\t}(x,y)} pa_1^x(1-a_1)^{y-1}[(x+1) (1-a_1)-ya_1],\\
\nonumber \partial_{a_2} \phi_{\t}(x,y) & = \ee^{-\phi_{\t}(x,y)}
(1-p)a_2^x(1-a_2)^{y-1}[(x+1) (1-a_2)-ya_2].
\end{align}
Since 
\[
  \ee^{-\phi_{\t}(x,y)}  pa_1^x(1-a_1)^{y-1}  \leq \frac1{a_1(1-a_1)},
\]
and
\[
  \ee^{-\phi_{\t}(x,y)}  (1-p)a_2^x(1-a_2)^{y-1}  \leq \frac1{a_2(1-a_2)},
\]
we can see that there exists a constant $B$ such that
\begin{equation} \label{equa:majo_dot_phi}
  |\partial_{a_j}  \phi_{\t}(x,y)  |   \leq  \Big|  \frac{x+1}{a_j}  -
  \frac{y}{1-a_j}\Big| \leq B (x+1+y), \quad \mbox{for } j=1,2.
\end{equation}
Now,  we prove  that \eqref{eq:dotphi_stat_integr}  is  satisfied with
$q=3/2$.  From \eqref{equa:majo_dot_phi},  it is  sufficient  to check
that 
\[
\sum_{k\in\Z_+}k^{3}\pi_{\t}(k) = \sum_{x,y\in\Z_+} x^3 \tilde \pi_\t(x,y) = \sum_{x,y\in\Z_+} y^3 \tilde \pi_\t(x,y)<\infty,
\]
which is equivalent to 
\[
  \sum_{k\geq   3}   k(k-1)(k-2)\pi_{\t}(k)   = 6
  \Et \Big[\Big(\sum_{n\geq 1} \prod_{k=1}^n \rho_k \Big)^3 \Big]<\infty,
\]
where        the         last        equality        comes        from
Proposition~\ref{prop:moment_pi}. From Minkowski's inequality, we have 
\[
  \Et \Big[\Big(\sum_{n\geq 1} \prod_{k=1}^n \rho_k \Big)^3 \Big] 
  \leq
  \Big\{ \sum_{n \geq 1} \Big[\Et \Big( \prod_{k=1}^n \rho_k^3 \Big) \Big]^{1/3}
  \Big\}^{3} 
   = \Big\{\sum_{n \geq 1} [\Et( \rho_0^3 )]^{n/3} \Big\}^{3},
\]
where the right-hand side term is finite according to the additional
assumption    that    $\Et(\rho_0^3)<1$.    Since   the    bound    in
\eqref{equa:majo_dot_phi}  does not  depend on  $\theta$  and $\pi_\t$
possesses a finite third moment, the first part of condition
\eqref{eq:phi_dot_square_control} on the gradient vector is also satisfied. 

Now, we turn to~\eqref{eq:chaing_dot}. Noting that
\[
  \partial_{a_1}       Q_\t        (x,y)       =       \binom{x+y}{x}
  pa_1^x(1-a_1)^{y-1}[(x+1) (1-a_1)-ya_1],
\]
\[
  \partial_{a_2}       Q_\t        (x,y)       =       \binom{x+y}{x}
  (1-p)a_2^x(1-a_2)^{y-1}[(x+1) (1-a_2)-ya_2], 
\]
\[
\sum_{y=0}^{\infty}y \binom{x+y} x a^{x+1}(1-a)^y= (x+1)\frac{1-a}{a},
\quad \forall x \in \Z_+, \ \forall a \in (0,1), 
\]
and using \eqref{eq:negbinom} yields \eqref{eq:chaing_dot}.

The second order derivatives of $\phi_{\t}$ are given by 
 \begin{align*}
  \partial_{p}^2 \phi_{\t}(x,y) & = - [\partial_{p} \phi_{\t} (x,y)]^2, \\
\partial_{p}   \partial_{a_1}  \phi_{\t}(x,y)   &   =  [\partial_{a_1}
\phi_{\t} (x,y)] \times \left( \frac 1 p - \partial_{p} \phi_{\t}
  (x,y)\right), \\
\partial_{a_1}\partial_{a_2}  \phi_{\t}(x,y)   &  =  - [ \partial_{a_1}
\phi_{\t}(x,y) ]\times [\partial_{a_2} \phi_{\t}(x,y) ], \\
\partial_{a_1}^2 \phi_{\t} (x,y) & = [\partial_{a_1} \phi_{\t}(x,y)]
\times \Big[ - \partial_{a_1} \phi_{\t}(x,y) +\frac x {a_1} -\frac{y-1}{1-a_1}\\
& -\frac{x+1+y}{(x+1)(1-a_1)-ya_1} 
\Big],
\end{align*}
and similar formulas for $a_2$ instead of $a_1$.  The second part of~\eqref{eq:phi_dot_square_control}  on  the  Hessian matrix  thus
follows from the previous expressions 
combined     with~\eqref{eq:bound_dot_phi_p},~\eqref{equa:majo_dot_phi}
and the existence of a second order moment for $\pi_{\t}$.
\end{proof}

\begin{prop} \label{prop:NonDegeEx2}
In the framework of Example~\ref{ex:deuxpts_3param}, the covariance matrix
$\Sigma_\t$        is         positive        definite,        namely
Assumption~\ref{assu:nonsingular} is satisfied. 
\end{prop}

\begin{proof}[Proof of Proposition~\ref{prop:NonDegeEx2}]
We have
\[
\Et[\omega_0^{x+1}(1-\omega_0)^y]=
pa_1^{x+1}(1-a_1)^y + (1-p)a_2^{x+1}(1-a_2)^y.
\]
The determinant  of $\Big(\partial_\t \Et[\omega_0^{k+1}]\Big)_{k=0,1,2}$
is given by
\[
\begin{array}{|ccc|}
a_1 - a_2 & a_1^2 - a_2^2 & a_1^3 - a_2^3 \\
p & 2 p a_1 & 3 p a_1^2 \\
(1-p) & 2(1-p) a_2 & 3 (1-p) a_2^2
\end{array}
\]
which can be rewritten as
\[
  p(1-p)(a_1-a_2)^4.
\] 
As we have $a_1 \neq  a_2$ and $ p \in (0,1)$, this determinant is non zero and this 
completes the proof, thanks to Proposition~\ref{prop:Sigma_non_zero}.
\end{proof}

Thanks             to            Theorem~\ref{theo:CLT}            and
Propositions~\ref{prop:2pts3paraNormal}      and~\ref{prop:NonDegeEx2},
under the  additional assumption that  $\Et(\rho_0^3)<1$, the sequence
$\{\sqrt{n}    (\widehat   \t_n    -    \t^\star)\}$   converges    in
$\PPs$-distribution  to  a  non  degenerate centered  Gaussian  random
vector. 

\subsection{Environment with Beta distribution} \label{sect:ex3}

\begin{exam} \label{ex:beta}
We let $\nu$ be  a Beta distribution with parameters $(\alpha,\beta)$,
namely
\[
\dd\nu(a) = \frac 1 {\Beta(\alpha,\beta)} a^{\alpha
  -1}(1-a)^{\beta -1} \dd a, \quad
   \Beta(\alpha,\beta)   =   \int_0^1   t^{\alpha
  -1}(1-t)^{\beta -1}\dd t.
\]
Here,  the  unknown parameter  is  $\theta=(\alpha,\beta) \in  \Theta$
where $\Theta$ is a compact subset of 
\[
  \{ (\alpha,\beta) \in (0,+ \infty)^2 \, : \, \alpha > \beta+1 \}.
\]
As $\E^{\t}(\rho_0)=\beta/(\alpha-1)$, the  constraint $\alpha > \beta
+1$        ensures       that       points~\hyperref[assu:transience]{$i)$}
and~\hyperref[assu:ballistic]{$ii)$}                                 in
Assumption~\ref{assu:consistance} are satisfied. 
\end{exam}

In the framework of Example~\ref{ex:beta}, we have
\begin{equation}
\phi_\t(x,y) = \log \frac{\Beta(x+1+\alpha,y+\beta)}{\Beta(\alpha,\beta)}
\end{equation}
and
\begin{align*} \label{eq:ell_beta}
  \ell_n(\theta) &= -n  \log \Beta(\alpha,\beta)+ \sum_{x=0}^{n-1} \log
  \Beta(L_{x+1}^n + \alpha+1,L_x^n + \beta)\\
&=   
 \sum_{x=0}^{n-1} \log \frac{(L_{x\!+\!1}^n \!+ \!\alpha)(L_{x\!+\!1}^n\! +\! \alpha\!-\!1)\ldots \alpha
 \times (L_{x}^n \!+\! \beta\!-\!1)(L_{x}^n \!+\! \beta\!-\!2)\ldots \beta}
 {(L_{x+1}^n \!+\!L_{x}^n \!+\! \alpha \!+\!\beta\!-\!1)(L_{x+1}^n\! +\!L_{x}^n \!+\! \alpha \!+\!\beta\!-\!2)\ldots
 (\alpha+ \beta)}.
 \end{align*}
In this  case, \cite{Comets_etal} proved that  $\htet$ is well-defined
and consistent in probability.
We now establish that  the assumptions needed for asymptotic normality
are also satisfied in this case.

\begin{prop} \label{prop:betaNormal}
In  the  framework  of  Example~\ref{ex:beta},  Assumptions~\ref{assu:diff}
to~\ref{assu:uniform} are satisfied.  
\end{prop}

\begin{proof}[Proof of Proposition~\ref{prop:betaNormal}]
Relying  on   classical  identities  on  the  Beta
function, it may be seen after some computations that 
\[
  \phi_\t(x,y)  =  \sum_{k=0}^x  \log  (k+\alpha)  +  \sum_{k=0}^{y-1}
  \log(k+\beta) - \sum_{k=0}^{x+y} \log(k+\alpha+\beta), 
\]
where a sum over an empty set of indices is zero. 
As a consequence, we obtain
\begin{align}
\label{equa:PartialAlpha}    \partial_\alpha    \phi_\t    (x,y)    &=
\sum_{k=0}^x         \frac1{k+\alpha}        -        \sum_{k=0}^{x+y}
\frac1{k+\alpha+\beta}\\ 
\nonumber &= \sum_{k=0}^x \frac{\beta}{(k+\alpha)(k+\alpha+\beta)} - 
   \sum_{k=1}^{y} \frac{1}{k+x+\alpha+\beta}.
\end{align}
The fact  that $\Theta$  is a compact  set included  in $(0,+\infty)^2$
yields the  existence of a  constant $A$ independent of  $\theta$, $x$
and $y$ such that both 
\begin{align*}
\sum_{k=0}^x       \frac{\beta}{(k+\alpha)(k+\alpha+\beta)}      &\leq
\sum_{k=0}^{+\infty} \frac{\beta}{(k+\alpha)(k+\alpha+\beta)} \leq A,
\end{align*}
and
\begin{align*}
\sum_{k=1}^{y}    \frac{1}{k+x+\alpha+\beta}    &\leq   \sum_{k=1}^{y}
\frac{1}{k+\alpha+\beta} \leq A \log(1+y) . 
\end{align*}
The same holds for $\partial_\beta  \phi_\t(x,y)$. Hence, we have
\begin{equation} \label{equa:majo_dot_phi_beta}
|\partial_\alpha \phi_\t (x,y)| \leq A' \log(1+y) \quad \mbox{and} \quad
|\partial_\beta \phi_\t (x,y)| \leq A'\log(1+x),
\end{equation}
for some  positive constant  $A'$. Since there  exists a  constant $B$
such that for any integer $x$ 
\[
  \log(1+x)\leq B \sqrt[4]{x},
\]
we deduce from \eqref{equa:majo_dot_phi_beta} that there exists $C>0$ such that
\begin{equation} \label{equa:majo_dot_phi_beta2}
  |\partial_\alpha\phi_{\t}(x,y)|^{2q}\leq Cy 
  \quad \mbox{and} \quad
  |\partial_\beta\phi_{\t}(x,y)|^{2q}\leq Cx,
\end{equation}
where $q=2$.  From Proposition~\ref{prop:moment_pi}, we know that
$\pi_\t$  possesses   a  finite   first  moment,  and   together  with~\eqref{equa:majo_dot_phi_beta2},     this     is    sufficient     for~\eqref{eq:dotphi_stat_integr}  to  be satisfied.  Since  the bound  in~\eqref{equa:majo_dot_phi_beta2} does not depend on $\theta$, the first
part of condition~\eqref{eq:phi_dot_square_control}  on  the  gradient  vector  is  also
satisfied. 

The second order derivatives of $\phi_{\t}$ are given by 
 \begin{align*}
  \partial_{\alpha}^2  \phi_{\t}(x,y)   &  =  -\sum_{k=0}^x   \frac  1
  {(k+\alpha)^2} +\sum_{k=0}^{x+y} \frac 1 {(k+\alpha+\beta)^2} , \\
\partial_{\alpha}  \partial_{\beta} \phi_{\t}(x,y)  & =\sum_{k=0}^{x+y}
\frac 1 {(k+\alpha+\beta)^2} ,
\end{align*}
and  similar formulas  for  $\beta$ instead  of  $\alpha$.  Thus,  the
second part of condition~\eqref{eq:phi_dot_square_control}  for  the  Hessian  matrix
follows        by       arguments        similar        to       those
establishing the first part of~\eqref{eq:phi_dot_square_control} for the gradient
vector.

Now, we prove that it is  possible to exchange the order of derivation
and summation to get~\eqref{eq:chaing_dot}. To do so, we prove that
\begin{equation} \label{equa:ConvNorm}
\mbox{the series } \sum_y \| \dot Q_\t (x,y) \| \mbox{ converges uniformly in $\t$}
\end{equation}
for any integer $x$. 
Define $\theta_0=(\alpha_0,\beta_0)$ with
\[
  \alpha_0 = \inf (\proj_1(\Theta)) \quad \mbox{and} \quad 
\beta_0 = \inf (\proj_2(\Theta)),
\]
where $\proj_i, i=1,2$ are the two projectors on the coordinates. 
Note that $\theta_0$ does not necessarily belong to $\Theta$. However, 
it still belongs to the ballistic region $\{\alpha
>\beta+1\}$. For  any $a \in (0,1)$  and any integers $x$  and $y$, we
have 
\[
  a^{x+1+\alpha-1}(1-a)^{y+\beta-1} \leq a^{x+1+\alpha_0-1}(1-a)^{y+\beta_0-1},
\]
which yields
\[
\Beta(x+1+\alpha,y+\beta) \leq \Beta(x+1+\alpha_0,y+\beta_0),
\]
as well as
\[
  Q_{\t}(x,y) \leq \frac{\Beta(\alpha_0,\beta_0)}{\Beta(\alpha,\beta)}
  Q_{\theta_0}(x,y). 
\] 
Using the fact that the beta function is continuous on the compact set
$\Theta$ yields the existence of a constant $C$ such that
\[
  Q_{\t}(x,y) \leq C Q_{\theta_0}(x,y),
\]
for  any integers  $x$  and $y$.  Now  recall that  $\dot Q_\t(x,y)  =
Q_\t(x,y) \dot  \phi_\t (x,y)$. Hence,  using the last  inequality and
\eqref{equa:majo_dot_phi_beta2}, it is sufficient to prove that 
\begin{equation} \label{equa:MomentQ}
\sum_y y Q_{\theta_0}(x,y) < \infty,
\end{equation}
to get \eqref{equa:ConvNorm}. We have
\[
  \sum_x \Big(\sum_y y  Q_{\theta_0}(x,y) \Big) \pi_{\theta_0}(x) =
  \sum_y y \pi_{\theta_0}(y) < \infty,
\]
where the last inequality comes  from the fact that $\theta_0$ lies in
the ballistic region and thus $\pi_{\theta_0}$
possesses  a  finite  first  moment.  Since $\pi_{\theta_0}(x) >0$ for any integer~$x$, we deduce that~\eqref{equa:MomentQ}  is
satisfied for any integer  $x$ which proves that \eqref{equa:ConvNorm}
is satisfied.
\end{proof}

\begin{prop} \label{prop:NonDegeEx3}
  In the framework of Example~\ref{ex:beta}, the covariance matrix
$\Sigma_\t$        is         positive        definite,        namely
Assumption~\ref{assu:nonsingular} is satisfied. 
\end{prop}

\begin{proof}[Proof of Proposition~\ref{prop:NonDegeEx3}]
One easily checks that
\[
\dot \phi_\t (x,x)= 
\begin{pmatrix}
\frac{1}{\alpha +x}+\frac{1}{\alpha +x-1}+ \dots + \frac{1}{\alpha}
-   \frac{1}{\alpha   +\beta+2x}-\frac{1}{\alpha  +\beta+2x-1}-\dots-
\frac{1}{\alpha+\beta} 
\\ \\
\frac{1}{\beta +x-1}+\frac{1}{\beta +x-2}+\dots+ \frac{1}{\beta}
-   \frac{1}{\alpha   +\beta+2x}-\frac{1}{\alpha  +\beta+2x-1}-\dots-
\frac{1}{\alpha+\beta} 
\end{pmatrix} .
\]
Hence, $\dot \phi_\t (0,0)$ is collinear to
 $(\beta,-\alpha)^\intercal$ and
 $\dot \phi_\t (x,x) \to  (-\log 2,-\log 2)^\intercal$ 
 as $x \to \infty$. This shows that $\dot \phi_\t (x,x), x \in \Z_+$,
 spans the whole space, and Proposition~\ref{prop:Sigma_non_zero} applies.
\end{proof}

Thanks             to            Theorem~\ref{theo:CLT}            and
Propositions~\ref{prop:betaNormal}    and~\ref{prop:NonDegeEx3},   the
sequence  $\{\sqrt{n}  (\widehat  \t_n  -  \t^\star)\}$  converges  in
$\PPs$-distribution  to  a  non  degenerate centered  Gaussian  random
vector.


\section{Asymptotic normality} \label{sect:CLT}
We now establish the asymptotic normality of $\widehat \t_n$ stated in
Theorem~\ref{theo:CLT}.
The     most      important     step     lies      in     establishing
Theorem~\ref{thm:CLT_dotphi} that states 
a CLT for  the gradient vector
of the criterion $\ell_n$ (see Section~\ref{sec:CLT_dotphi}). To obtain the asymptotic 
 normality of  $\widehat \t_n$ from the  former CLT, we make  use of a
 uniform weak law of large numbers (UWLLN) in Section~\ref{sec:AN}. The proof of the UWLLN is contained in Section~\ref{sec:approx} and establishes Proposition~\ref{prop:sigma_approx} giving a way to approximate the Fisher information matrix. Finally Section~\ref{sec:nonzero} establishes the proof of Proposition~\ref{prop:Sigma_non_zero} stating a condition under which
the Fisher information matrix is non singular.


\subsection{A central limit theorem for the gradient of the criterion}
\label{sec:CLT_dotphi}
In this  section, we prove  Theorem~\ref{thm:CLT_dotphi}, that is,
the existence of a CLT for the score vector sequence $\dot \ell_n(\ts)$. Note that
according to~\eqref{eq:logvraisannealed}, we have 
\begin{equation}
\frac   1    {\sqrt{n}   }\dot   \ell_n(\ts)    \sim   \frac1   {\sqrt
  n}\sum_{k=0}^{n-1} \dot\phi_{\ts} ( Z_k, Z_{k+1}) ,
\end{equation}
where  $(Z_k)_{0\le k  \le n}$  is  the Markov  chain introduced  in
Section~\ref{sect:BP}. 
First, note that under Assumption~\ref{assu:smoothness} this quantity 
is  integrable and centered  with respect  to $\PPs$.   Indeed, recall
that $\dot\phi_\t(x,y)=\dot Q_\t(x,y)/Q_\t(x,y),$ 
 thus  we can write for all $x \in \Z_+$, 
\begin{align}\label{eq:martingale}
 \EEs(\dot\phi_{\ts}  (Z_k,  Z_{k+1}) | Z_k=x)  
  &= 
  \sum_{y \in \Z_+} \frac  {\dot Q_{\ts}(x,y)} {Q_{\ts}(x,y)} Q_{\ts}(x,y)
= \partial_\t \Big(\sum_{y\in \Z_+} Q_{\t}(x,y) \Big) \Big|_{\t=\ts} \nonumber \\
&= \partial_\t (1) \Big|_{\t=\ts} =0,
\end{align}
where  we  have  used  \eqref{eq:chaing_dot} to  interchange  sum  and
derivative. Then, 
\[
 \EEs(\dot\phi_{\ts} (Z_k, Z_{k+1}))=0.
\]
Now, we rely on a  CLT for centered square-integrable martingales, see
Theorem 3.2 in~\cite{Hall_Heyde}. We introduce the quantities 
\[
\forall  1\le k \le n, \quad 
U_{n,k}=\frac  1  {\sqrt  n}  \dot\phi_{\ts}  (Z_{k-1},  Z_{k})  \quad
\mbox{ and } \quad S_{n,k} =\sum_{j=1}^k U_{n,j}, 
\]
as         well        as        the         natural        filtration
$\mathcal{F}_{n,k}=\mathcal{F}_{k}:=\sigma(Z_j, j\le k)$.  
According  to~\eqref{eq:martingale}, $(S_{n,k},  1\le k  \le  n, n\ge
1)$  is a  martingale array  with differences  $U_{n,k}$. It  is also
centered and square integrable from 
 Assumption~\ref{assu:smoothness}.  Thus   according  to  Theorem  3.2
 in~\cite{Hall_Heyde}    and    the    Cramér-Wold   device  \cite[see, e.g.][p. 48]{bill}, as soon as we have 
 \begin{align}
 & \max_{1\le i  \le n} \|U_{n,i}\| \xrightarrow[n\to+\infty]{}
 0 \mbox{ in } \PPs\mbox{-probability}, \label{eq:controle_max}\\ 
&  \sum_{i=1}^n U_{n,i}  U_{n,i}^\intercal \xrightarrow[n\to+\infty]{}
\Sigma_{\ts}                 \mbox{                in                }
\PPs\mbox{-probability}, \label{eq:controle_variance}\\ 
\mbox{     and    }     &    \left (     \EEs(\max_{1\le    i     \le    n}
\| U_{n,i}^{\phantom{\intercal}} U_{n,i}^\intercal \| )\right )_{n\in\Z_+}     \mbox{    is     a    bounded
  sequence}, \label{eq:controle_esp_max} 
 \end{align}
with $\Sigma_{\ts}$ a deterministic and finite covariance matrix, then
the  sum $S_{n,n}$ converges  in distribution  to a  centered Gaussian
random  variable with covariance  matrix $\Sigma_{\ts}$,  which proves
Theorem~\ref{thm:CLT_dotphi}.  Now, the 
convergence~\eqref{eq:controle_variance}  is a  direct  consequence of
the ergodic theorem  stated in Proposition~\ref{prop:lln}. Moreover
the limit $ \Sigma_{\ts}$ is given by \eqref{eq:Sigma} 
and is finite according to Assumption~\ref{assu:smoothness}. 
Note     that     more      generally,     the     ergodic     theorem
(Proposition~\ref{prop:lln})               combined               with
Assumption~\ref{assu:smoothness}    implies    the   convergence    of
$(\sum_{1\leq i \leq n}\|U_{n,i}\|^2)_n$ to a finite deterministic limit,
$\PPs$-almost    surely    and    in   $\mathbb{L}_1(\PPs)$.     Thus,
condition~\eqref{eq:controle_esp_max}      follows      from      this
$\mathbb{L}_1(\PPs)$-convergence, combined with the bound 
\[
\EEs(\max_{1\le  i  \le  n}  \|U_{n,i}U_{n,i}^\intercal\|)  \le
         \sum_{i=1}^n
\EEs(\|U_{n,i}\|^2 ). 
\]
Finally, condition~\eqref{eq:controle_max} is obtained by writing that
for any $\eps >0$ and any $q>1$, we have
\begin{align*}
\PPs( \max_{1\le i \le n} \|U_{n,i}\| \ge \eps ) &=
\PPs(\max_{1\le  i \le  n}  \|\dot\phi_{\ts}(Z_{i-1},Z_i)\| \ge
\eps \sqrt{n}) \\
&\le  \frac   1  {n^q  \eps^{2q}}   \EEs(  \max_{1\le  i   \le  n}
\|\dot\phi_{\ts}(Z_{i-1},Z_i)\|^{2q} ) \\
&\le       \frac       1       {n^q      \eps^{2q}}       \sum_{i=1}^n
\EEs( \|\dot\phi_{\ts}(Z_{i-1},Z_i)\|^{2q} ) , 
\end{align*}
where the first inequality is Markov's inequality.
By  using  again   Assumption~\ref{assu:smoothness}  and  the  ergodic
theorem  (Proposition~\ref{prop:lln}),  the  right-hand side  of  this
inequality converges to zero whenever $q>1$. This achieves the proof.


\subsection{Approximation of the Fisher information} \label{sec:approx}

We now turn to the proof of Proposition~\ref{prop:sigma_approx}. Under Assumption~\ref{assu:uniform},  the following local uniform convergence holds: there exists a neighborhood $\Vs$ of $\ts$ such that
\begin{equation} \label{equa:ULLN}
\sup_{\t \in \Vs} \left \| \frac  1n   \sum_{x=0}^{n-1} \ddot \phi_{\t} (L_{x+1}^n,L_x^n) - \tilde \pi_\ts (\ddot \phi_{\t}) \right \| \xrightarrow[n  \to  \infty]{} 0 \quad \mbox{in $\PPs$-probability}.
\end{equation}
This could be verified by the same arguments as in the proof of the standard uniform law of large numbers \cite[see Theorem 6.10 and its proof in Appendix 6.A in][]{bierens} where the ergodic theorem stated in our Proposition~\ref{prop:lln} plays the role of the weak law of large numbers for a random sample in the former reference. Indeed, let $\ddot \phi_\t ^{(i,j)}$ represent the element at the $i$th row and $j$th column of the matrix $\ddot \phi_\t$. Under Assumption~\ref{assu:uniform}, there exists a neighborhood $\mathcal{V}(\ts)$ of $\ts$ such that 
\[
\tilde  \pi_\ts \Big( \sup_{\t \in \mathcal{V}(\ts)} \left | \ddot \phi_{\t}^{(i,j)} \right | \Big) < +\infty, \quad \mbox{for any $1 \leq i,j \leq d$},
\]
which implies that
\[
\tilde  \pi_\ts \Big( \sup_{\t \in \mathcal{V}(\ts)} \ddot \phi_{\t}^{(i,j)}  \Big) < +\infty \quad \mbox{and} \quad \tilde  \pi_\ts \Big( \inf_{\t \in \mathcal{V}(\ts)} \ddot \phi_{\t}^{(i,j)}  \Big) > -\infty,
\]
for any $1 \leq i,j \leq d$. Furthermore, under Assumption~\ref{assu:smoothness}, the map $\t \mapsto \ddot \phi_{\t}^{(i,j)}$ is continuous for any $1 \leq i,j \leq d$ and according to Theorem 6.10 in  \cite{bierens}, there exists a neighborhood $\Vs$ of $\ts$ such that
\[
\sup_{\t \in \Vs} \left | \frac  1n   \sum_{x=0}^{n-1} \ddot \phi_{\t}^{(i,j)} (L_{x+1}^n,L_x^n) - \tilde \pi_\ts (\ddot \phi_{\t}^{(i,j)}) \right | \xrightarrow[n  \to  \infty]{} 0 \quad \mbox{in $\PPs$-probability} 
\]
for any $1 \leq i,j \leq d$. This implies \eqref{equa:ULLN}. The latter combined with the convergence in $\PPs$-probability of $\htet$ to $\ts$ yields \eqref{equa:Sigma_hat}.


\subsection{Proof of asymptotic normality}
\label{sec:AN}

Our estimator~$\widehat  \theta_n$  maximizes  the function  $\t  \mapsto
\ell_n(\t)=\sum_{x=0}^{n-1}  \phi_\t(L_{x+1}^n,L_{x}^n)$.  As a consequence, under Assumption~\ref{assu:diff},  we have
\begin{equation} \label{equa:max_local}
  \dot \ell_n(\htet) = \sum_{x=0}^{n-1} \dot \phi_{\htet}(L_{x+1}^n,L_{x}^n)=0.
\end{equation}
Using a Taylor expansion in a neighborhood of $\ts$, there exists a random $\tilde \t_n$ such that $\| \tilde \t_n - \ts \| \leq \| \htet - \ts \|$ and
\begin{equation} \label{equa:taylor}
  \frac 1 {\sqrt{n}} \dot \ell_n( \htet) = \frac 1 {\sqrt{n}} \dot \ell_n( \ts) + \frac 1 n \ddot \ell_n( \tilde \t_n) \cdot \sqrt{n} (\htet - \ts).
\end{equation}
Combining \eqref{equa:max_local} and~\eqref{equa:taylor} yields
\[
  \frac 1 n \ddot \ell_n( \tilde \t_n) \cdot \sqrt{n} (\htet - \ts) = -  \frac 1 {\sqrt{n}} \dot \ell_n( \ts).
\]
Using~\eqref{equa:ULLN} and the convergence in $\PPs$-probability of $\htet$  to   $\ts$ yields
\[
  \tilde \pi_\ts(  \ddot \phi_\ts) \sqrt{n}  (\htet - \ts) =  -\frac 1
  {\sqrt{n}} \dot \ell_n( \ts) (1+ o_{\PPs}(1) ),
\]
where   $o_{\PPs}(1)$   is   a   remainder   term   that   converges   in
$\PPs$-probability to 0. 
If   we   moreover  assume   that   the   Fisher  information   matrix
$\Sigma_{\ts}=- \tilde \pi_\ts (\ddot \phi_\ts)$ is non singular, then we have 
\begin{equation}
  \label{eq:LAN}
\sqrt{n}(\widehat    \theta_n   -\ts)    = \Sigma_{\ts}^{-1}\frac    1   {\sqrt
  n}\sum_{x=0}^{n-1} \dot \phi_{\ts}(L_{x+1}^n,L_{x}^n) (1+o_{\PPs}(1)).
\end{equation}


Finally,  combining~\eqref{eq:LAN}  with Theorem~\ref{thm:CLT_dotphi},  we
obtain the convergence in $\PPs$-distribution of $\sqrt{n}(\widehat
\theta_n -\ts) $ to a centered Gaussian random vector with covariance matrix
$\Sigma_{\ts}^{-1} \Sigma_\ts \Sigma_{\ts}^{-1}=\Sigma_\ts^{-1}$.


\subsection{Non degeneracy of the Fisher information}\label{sec:nonzero}

We now turn to the proof of Proposition~\ref{prop:Sigma_non_zero}. 
Let us consider a deterministic vector $u\in \R^d$. We have
\[
u^\intercal \Sigma_{\t} u = \tilde \pi_\t (\|u^\intercal \dot\phi_{\t} \|^2 ).
\]

We  recall  that  according to  Proposition~\ref{prop:moment_pi},  the
invariant probability  measure  $\pi_{\t}$ is  positive  as  well  as $  \tilde
\pi_{\t}$.  As a consequence, the quantity $u^\intercal \Sigma_{\t}
u$ is non negative and equals zero if and only if 
\[
\forall x,y\in \Z_+, \quad u^\intercal
\dot\phi_{\t}(x,y) =0 .
\]
Let us assume that the linear span in ${\mathbb R}^d$ of the gradient vectors 
$\dot \phi_{\t}(x,y), (x ,y) \in \Z_+^2$ is equal to the full
space,  or equivalently,  that
\[
{\rm                      Vect}\Big\{                      \partial_\t
\Et(\omega_0^{x+1}(1-\omega_0)^y) \, : \, (x, y)\in \Z_+^2
\Big\} ={\mathbb R}^d. 
 \]
Then,  the equality  $ u^\intercal  \dot\phi_{\t}(x,y) =0  $  for any
$(x,y)\in \Z_+^2$ implies $u=0$. This concludes the proof.


\section{Numerical performance} \label{sect:simus}
In~\cite{Comets_etal},  the authors  have  investigated the  numerical
performance of  the MLE and  obtained that this estimator  has better
performance than  the one proposed by~\cite{AdEn},  being less spread
out than  the latter. In this  section, we explore  the possibility to
construct confidence regions for  the parameter $\t$, relying on the
asymptotic         normality         result        obtained         in
Theorem~\ref{theo:CLT}.    From  Proposition~\ref{prop:sigma_approx},
the limiting covariance $\Sigma_{\ts}^{-1}$ may be approximated by the
inverse of the observed Fisher
  information matrix $\hat \Sigma_{n}$ defined by~\eqref{equa:Sigma_hat},
 and Slutsky's Lemma gives the convergence in distribution 
\[
\sqrt{n}  \hat   \Sigma_{n}^{1/2}(\htet  -\ts)  \xrightarrow [n\to
+\infty] {}\mathcal{N}_d (0,Id)  \mbox{ under } \PPs, 
\]
where 
$\mathcal{N}_d (0,Id)$ is  the centered and normalised $d$-dimensional
normal  distribution. When $d=1$, we thus  consider confidence intervals of the form 
\begin{equation}
  \label{eq:interval}
\mathcal{IC}_{\gamma,n} = \Big[\htet - \frac{q_{1-\gamma/2}}{\sqrt{n}\Sigma_{n}^{1/2}} ; \htet +\frac{q_{1-\gamma/2}}{\sqrt{n}\Sigma_{n}^{1/2}} \Big], 
\end{equation}
where  $1-\gamma$ is  the asymptotic  confidence level  and  $q_z$ the
$z$-th quantile of the standard normal one-dimensional distribution.
In higher dimensions ($d\ge 2$), the confidence regions are more generally built relying on the chi-square distribution, namely 
\begin{equation}
  \label{eq:region}
\mathcal{R}_{\gamma,n} = \{ \theta \in \Theta : n\| \hat 
\Sigma_{n}^{1/2}(\htet -\t) \|^2 \le \chi_{1-\gamma} \} , 
\end{equation}
where  $1-\gamma$ is  still the asymptotic  confidence level  and now $\chi_z$ is the
$z$-th quantile of the chi-square distribution with $d$ degrees of freedom $\chi^2(d)$. Note that the two definitions~\eqref{eq:interval} and~\eqref{eq:region} coincide when $d=1$. Moreover, the confidence region~\eqref{eq:region} is also given by 
\[
\mathcal{R}_{\gamma,n} = \{ \theta \in \Theta : n
(\htet -\t)^\intercal \hat \Sigma_{n}(\htet -\t)   \le \chi_{1-\gamma} \} . 
\]

We  present  three  simulation  settings corresponding  to  the  three
examples  developed  in  Section~\ref{sect:ex} and  already  explored
in~\cite{Comets_etal}.  For each of  the three
simulation  settings,  the  true   parameter  value  $\ts$  is  chosen
according  to  Table~\ref{tabl:theta_values}   and  corresponds  to  a
transient and ballistic random walk. We rely on 1000 iterations of
each of the following procedures.  For each setting and each iteration, we first generate
a  random environment  according  to  $\nu_\ts$ on  the  set of  sites
$\{-10^4, \dots, 10^4\}$.  Note that   we do  not use  the environment
values  for all the  $10^4$ negative  sites, since  only few  of these
sites are visited by the walk.  However this extra computation cost is
negligible. 
Then,  we  run  a  random   walk  in  this  environment  and  stop  it
successively     at     the     hitting    times     $T_n$     defined
by~\eqref{equa:HittingTime}, with $n \in \{10^3 k : 1\le k \le 10 \}$.  
For each stopping value $n$, we compute the estimators $\htet,\hat \Sigma_{n}$ 
and the confidence region $\mathcal{R}_{\gamma,n}$ for 
$\gamma= \{0.01;0.05;0.1\}$.

\begin{table}[h]
  \centering
  \begin{tabular}{|c|c|c|}
    \hline
Simulation & Fixed parameter & Estimated parameter \\
    \hline
Example~\ref{ex:deuxpoints} & $(a_1,a_2)=(0.4, 0.7)$ & $p^\star=0.3$ \\
    \hline
Example~\ref{ex:deuxpts_3param} & - & $( p^\star ,a_1^\star, a_2^\star)=(0.3,0.4,0.7)$\\ 
    \hline
Example~\ref{ex:beta} & -  & $(\alpha^\star, \beta^\star)=(5,1)$ \\
    \hline
  \end{tabular}
  \caption{Parameter values for each experiment.}
  \label{tabl:theta_values}
\end{table}

We  first explore  the  convergence of  $\hat  \Sigma_{n}$ when  $n$
increases. We  mention that the  true value $\Sigma_{\ts}$  is unknown
even   in  a   simulation   setting  (since   $\tilde  \pi_{\ts}$   is
unknown). Thus we can observe the convergence of $\hat  \Sigma_{n}$ with
$n$ but cannot assess any bias towards the true value $\Sigma_{\ts}$. The results
are                            presented                            in
Figures~\ref{fig:boxplot_Sighat_ex1},~\ref{fig:boxplot_Sighat_ex2}
and~\ref{fig:boxplot_Sighat_ex3}   corresponding  to   the   cases  of
Examples~\ref{ex:deuxpoints}, \ref{ex:deuxpts_3param} and~\ref{ex:beta}, respectively.  The
estimators appear  to converge when  $n$ increases and  their variance
also decreases as expected.  We mention that in the cases of Examples~\ref{ex:deuxpoints}
and~\ref{ex:deuxpts_3param}, we have $1\%$ and $1.3\%$
respectively of the total $10*1000$ experiments for which the
numerical   maximisation   of   the   likelihood  did   not   give   a
result  and  thus  for  which   we  could  not  compute  a  confidence
region.

\begin{figure}[h]
  \centering
  \includegraphics[height=12.5cm,width=6cm,angle=-90]{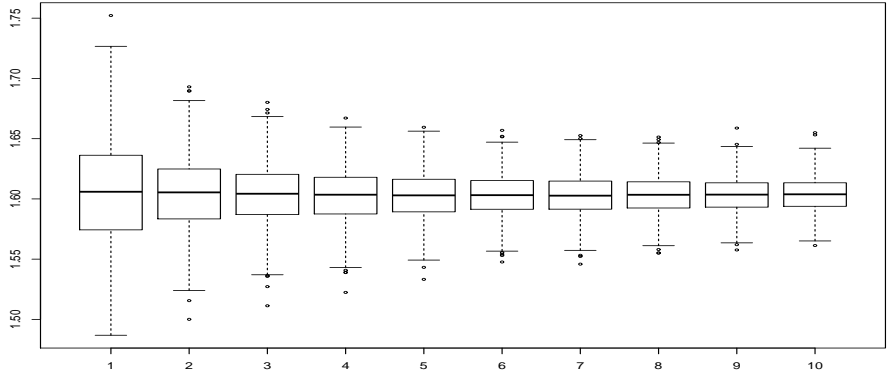}
  \caption{Boxplot  of the estimator  $\hat \Sigma_{n}$  obtained from
    $1000$ iterations and  for values $n$ ranging in  $\{10^3k : 1\le k
    \le 10\}$ in the case of Example~\ref{ex:deuxpoints}.}
  \label{fig:boxplot_Sighat_ex1}
\end{figure}

\begin{figure}[h]
  \centering
  \includegraphics[height=12.5cm,width=10cm,angle=-90]{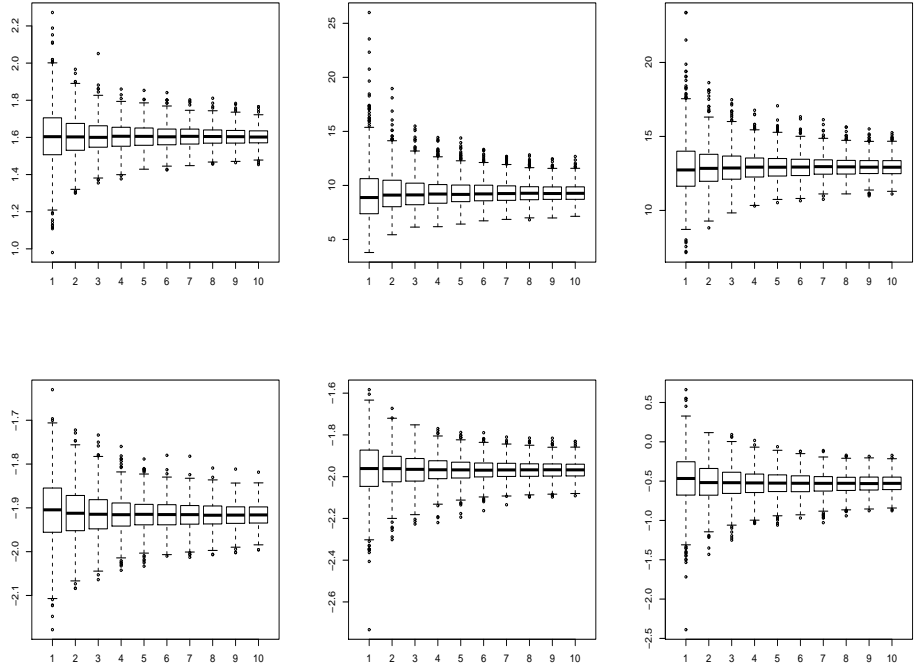}
  \caption{Boxplots of the values of the matrix $\hat \Sigma_{n}$ obtained from
    $1000$ iterations and  for values $n$ ranging in  $\{10^3k : 1\le k
    \le  10\}$ in  the case  of  Example~\ref{ex:deuxpts_3param}.  The
    parameter is ordered as 
    $\theta=(\theta_1,\theta_2,\theta_3)=(p,a_1,a_2)$ and the figure displays the
    values:   $\hat   \Sigma_n(1,1);   \hat   \Sigma_n(2,2)   ;   \hat
    \Sigma_n(3,3) ; \hat \Sigma_n(1,2); \hat \Sigma_n(1,3)$ and $\hat
    \Sigma_n(2,3) $, from left to right and top to
    bottom.}
  \label{fig:boxplot_Sighat_ex2}
\end{figure}

\begin{figure}[h]
  \centering
  \includegraphics[height=12.5cm,width=6cm,angle=-90]{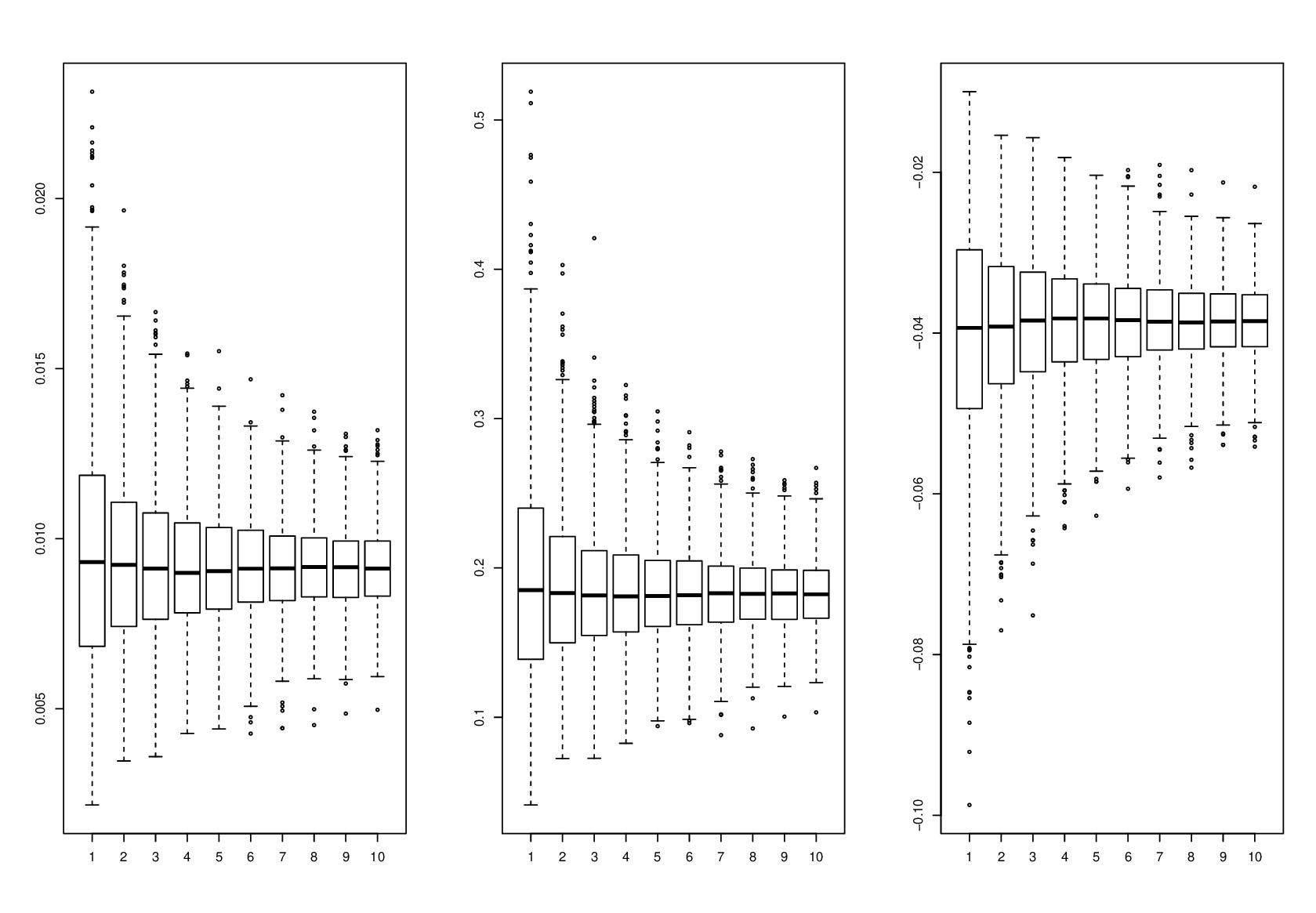}
  \caption{Boxplots of the values of the matrix $\hat \Sigma_{n}$ obtained from
    $1000$ iterations and  for values $n$ ranging in  $\{10^3k : 1\le k
    \le  10\}$ in  the case  of  Example~\ref{ex:beta}.  The
    parameter                is               ordered               as
    $\theta=(\theta_1,\theta_2)=(\alpha,\beta)$    and    the   figure
    displays the values: $\hat \Sigma_n(1,1); \hat \Sigma_n(2,2) $ and
    $\hat \Sigma_n(1,2) $, from left to right.}
  \label{fig:boxplot_Sighat_ex3}
\end{figure}

Now, we consider the empirical coverages obtained from our confidence
regions $\mathcal{R}_{\gamma,n}$ in the three examples and   with $\gamma\in \{0.01,0.05,0.1\}$ and $n$ ranging in  $\{10^3k : 1\le k
    \le      10\}$.     The      results     are      presented     in
    Table~\ref{tab:emp_cov}.  For the  three  examples, the  empirical
    coverages are very accurate. We also note that the accuracy does not
    significantly change when $n$ increases from $10^3$ to $10^4$.
As  a conclusion,  we  have shown  that  it is  possible to  construct
accurate confidence regions for the parameter value.

\begin{table}[ht]
\begin{center}
\begin{tabular}{|c|ccc|ccc|ccc|}
  \hline
&\multicolumn{3}{|c|}{Example~\ref{ex:deuxpoints} } & \multicolumn{3}{c|}{Example~\ref{ex:deuxpts_3param} } &\multicolumn{3}{c|}{Example~\ref{ex:beta} } \\
\cline{2-10}
$n$& 0.01 & 0.05 & 0.1 &0.01 & 0.05 & 0.1 &0.01 & 0.05 & 0.1 \\
  \hline
1000 & 0.994 & 0.952 & 0.899 & 0.992 & 0.953 & 0.909 & 0.977 & 0.942 & 0.901 \\ 
  2000 & 0.989 & 0.952 & 0.903 & 0.994 & 0.953 & 0.910 & 0.978 & 0.928 & 0.884 \\ 
  3000 & 0.988 & 0.942 & 0.901 & 0.990 & 0.938 & 0.886 & 0.981 & 0.940 & 0.889 \\ 
  4000 & 0.991 & 0.944 & 0.896 & 0.991 & 0.951 & 0.894 & 0.988 & 0.945 & 0.900 \\ 
  5000 & 0.990 & 0.942 & 0.896 & 0.993 & 0.942 & 0.891 & 0.986 & 0.941 & 0.883 \\ 
  6000 & 0.983 & 0.948 & 0.901 & 0.987 & 0.951 & 0.888 & 0.988 & 0.937 & 0.897 \\ 
  7000 & 0.986 & 0.950 & 0.900 & 0.992 & 0.951 & 0.900 & 0.986 & 0.942 & 0.898 \\ 
  8000 & 0.987 & 0.956 & 0.898 & 0.988 & 0.950 & 0.903 & 0.981 & 0.946 & 0.903 \\ 
  9000 & 0.990 & 0.959 & 0.913 & 0.990 & 0.949 & 0.893 & 0.985 & 0.939 & 0.901 \\ 
  10000 & 0.987 & 0.954 & 0.908 & 0.990 & 0.949 & 0.899 & 0.983 & 0.944 & 0.892 \\ 
   \hline
\end{tabular}
\caption{Empirical   coverages   of   $(1-\gamma)$  asymptotic   level
  confidence regions, for $\gamma  \in\{0.01, 0.05, 0.1\}$ and relying
  on 1000 iterations.}
\label{tab:emp_cov}
\end{center}
\end{table}


\
\\
\
\begin{merci}
 The authors warmly thank Francis Comets and Oleg Loukianov for 
 sharing many fruitful reflexions about this work.
\end{merci}


\bibliographystyle{chicago}
\bibliography{MAMA}

\end{document}